\documentclass{article}
\usepackage{amsmath,amssymb,amsthm}

\newtheorem{theorem}{Theorem}[section]
\newtheorem{proposition}[theorem]{Proposition}
\newtheorem{lemma}[theorem]{Lemma}

\theoremstyle{definition}
\newtheorem{definition}[theorem]{Definition}

\newtheorem{remark}[theorem]{Remark}

\title{Metric groups, unitary representations and continuous logic} 
\author{Aleksander Ivanov (Iwanow) }

\begin{document} 

\maketitle 


\begin{abstract} 
We describe how properties of metric groups and of unitary representations of 
metric groups can be presented in continuous logic.
In particular we find $L_{\omega_1 \omega}$-axiomatization of amenability.  
We also show that in the case of locally compact groups 
some uniform version of the negation of Kazhdan's property (T) 
can be viewed as a union of first-order axiomatizable classes. 
We will see when these properties are preserved under taking elementary substructures. 
\end{abstract}

%
%
%
%
%
%
%
%
%

\section{Introduction}
\label{intro} 
In this paper we study the behavior of amenability and 
Kazhdan's property {\bf (T)} under logical constructions. 
We view these tasks as a part of investigations of 
properties of basic classes of topological groups 
appeared in measurable and geometric group theory, 
see \cite{GriHarp}, \cite{HHM}, \cite{pestov16}. 
The fact that some logical constructions, for example ultraproducts, 
have become common in group theory, gives additional 
flavour for this topic.   
We concentrate on properties of metric groups which can be expressed in  continuous logic \cite{BYBHU}.   

Since we want to make the paper available for non-logicians,  
in Section 2 we briefly remind 
the reader some preliminaries of continuous logic.   
In Section 3 we apply it to amenability. 
In Section 4 we consider unitary representations of 
locally compact groups.

\section{Basic continuous logic} 

\subsection{Continuous structures} 
\label{S_1_1}

We fix a countable continuous signature 
$$
L=\{ d,R_1 ,...,R_k ,..., F_1 ,..., F_l ,...\}. 
$$  
Let us recall that a {\em metric $L$-structure} 
is a complete metric space $(M,d)$ with $d$ bounded by 1, 
along with a family of uniformly continuous operations $F_i$ 
on $M$ and a family of predicates $R_i$, 
i.e. uniformly continuous maps 
from appropriate $M^{k_i}$ to $[0,1]$.   
It is usually assumed that to a predicate symbol $R_i$ 
a continuity modulus $\gamma_i$ is assigned so that when 
$d(x_j ,x'_j ) <\gamma_i (\varepsilon )$ with $1\le j\le k_i$ 
the corresponding predicate of $M$ satisfies 
$$ 
|R_i (x_1 ,...,x_j ,...,x_{k_i}) - R_i (x_1 ,...,x'_j ,...,x_{k_i})| < \varepsilon . 
$$ 
It happens very often that $\gamma_i$ coincides with $id$. 
In this case we do not mention the appropriate modulus. 
We also fix continuity moduli for functional symbols. 
Each classical first-order structure can be considered 
as a complete metric structure with the discrete $\{ 0,1\}$-metric. 

By completeness,  continuous substructures of a continuous structure are always closed subsets. 

Atomic formulas are the expressions of the form $R_i (t_1 ,...,t_r )$, 
$d(t_1 ,t_2 )$, where $t_i$ are terms (built from functional $L$-symbols). 
In metric structures they can take any value from $[0,1]$.   
{\em Statements} concerning metric structures are usually 
formulated in the form 
$$
\phi = 0 
$$ 
(called an $L$-{\em condition}), where $\phi$ is a {\em formula}, 
i.e. an expression built from 
0,1 and atomic formulas by applications of the following functions: 
$$ 
x/2  \mbox{ , } x\dot- y= \mathsf{max} (x-y,0) \mbox{ , } \mathsf{min}(x ,y )  \mbox{ , } \mathsf{max}(x ,y )
\mbox{ , } |x-y| \mbox{ , } 
$$ 
$$ 
\neg (x) =1-x \mbox{ , } x\dot+ y= \mathsf{min}(x+y, 1) \mbox{ , } \mathsf{sup}_x \mbox{ and } \mathsf{inf}_x . 
$$ 
A {\em theory} is a set of $L$-conditions without free variables 
(here $\mathsf{sup}_x$ and $\mathsf{inf}_x$ play the role of quantifiers). 
   
It is worth noting that any formula is a $\gamma$-uniformly continuous 
function from the appropriate power of $M$ to $[0,1]$, 
where $\gamma$ is the minimum of continuity moduli of $L$-symbols 
appearing in the formula. 

The condition that the metric is bounded by $1$ is not necessary. 
It is often assumed that $d$ is bounded by some rational number $d_0$. 
In this case the (truncated) functions above are appropriately modified.    

We sometimes replace conditions of the form $\phi \dot{-} \varepsilon =0$ 
where $\varepsilon \in [0,d_0]$ by more convenient expressions $\phi \le \varepsilon$. 

In several places of the paper we use continuous 
$L_{\omega_1 \omega}$-logic. 
It extends the first-order logic by new connectives applied to 
countable families of formulas : $\bigvee$ is the infinitary 
$\mathsf{min}$ and $\bigwedge$ corresponds to the infinitary $\mathsf{max}$.  
When we apply these connectives we only demand that the formulas of the family all obey the same continuity modulus, 
see \cite{BYI}.

\subsection{ Metric groups} 
Below we always assume that our metric groups are 
continuous structures with respect 
to bi-invariant metrics (see \cite{BYBHU}). 
This exactly means that $(G,d)$ is 
a complete metric space and $d$ is bi-invariant. 
Note that the continuous logic approach 
takes weaker assumptions on $d$. 
Along with completeness it is only assumed 
that the operations  of a structure 
are uniformly continuous with respect to $d$. 
Thus it is worth noting here that 
\begin{itemize} 
\item any group which is a continuous structure has 
an equivalent bi-invariant metric. 
\end{itemize} 
Indeed assuming that $(G,d)$ is a continuous metric group 
which is not discrete one can apply the following 
function: 
$$ 
d^* (x,y) = sup_{u,v} d(u\cdot x\cdot v, u\cdot y\cdot v). 
$$ 
See Lemma 2 and Proposition 4 of \cite{sasza2} for further discussions concerning  this observation. 

\subsection{The approach} 
In Section 3 we apply the recent paper \cite{SchneiderThom} 
for $L_{\omega_1 \omega}$-axiomatization of amenability/non-amenability  
of metric groups.  
The case of property {\bf (T)} looks slightly more complicated, 
because unbounded metric spaces are involved in the definition. 
 
Typically unbounded metric spaces are considered in 
continuous logic as many-sorted structures of 
$n$-balls of a fixed point of the space ($n\in \omega$). 
Section 15 of \cite{BYBHU} contains nice examples of 
such structures. 
If the action of a bounded metric group $G$ is isometric 
and preserves these balls we may consider 
the action as a sequence of binary operations where  
the first argument corresponds to $G$.  
In such a situation one just fixes 
a sequence of continuity moduli for $G$ 
(for each $n$-ball). 
We will see in Section 4 that this approach  
works sufficiently well for the negation of property ${\bf (T)}$ 
(non-${\bf (T)}$)  in the class of locally compact groups.  

It is well-known that a locally compact group with property 
{\bf (T)} is amenable if and only if it is compact. 
Thus it is natural to consider these properties together. 

\subsection{Uniform continuity} 
Actions of metric groups which can be analyzed by 
tools of continuous logic must be uniformly continuous 
for each sort appearing in the presentation of the space by metric balls. 
This slightly restricts the field of applications.

\subsection{Hilbert spaces in continuous logic } \label{hs} 
We treat a Hilbert space over $\mathbb{R}$ 
exactly as in Section 15 of \cite{BYBHU}. 
We identify it with a many-sorted metric structure 
$$
(\{ B_n\}_{n\in \omega} ,0,\{ I_{mn} \}_{m<n} ,
\{ \lambda_r \}_{r\in\mathbb{R}}, +,-,\langle \rangle ),
$$
where $B_n$ is the ball of elements of norm $\le n$, 
$I_{mn}: B_m\rightarrow B_n$ is the inclusion map, 
$\lambda_{r}: B_m\rightarrow B_{km}$ is scalar 
multiplication by $r$, with $k$ the unique integer 
satisfying $k\ge 1$ and $k-1 \le |r|<k$; 
furthermore, $+,- : B_n \times B_n \rightarrow B_{2n}$ 
are vector addition and subtraction and 
$\langle \rangle : B_n \rightarrow [-n^2 ,n^2 ]$ 
is the predicate of the inner product. 
The metric on each sort is given by 
$d(x,y) =\sqrt{ \langle x-y, x-y \rangle }$.   
For every operation the continuity modulus is standard.  
For example in the case of $\lambda_r$ this is $\frac{z}{|r|}$. 
Note that in this version of continuous logic 
we do not assume that the diameter of a sort is bounded by 1. 
It can become any natural number.  

Stating existence of infinite approximations of orthonormal bases 
(by a countable family of axioms, see Section 15 of \cite{BYBHU}) 
we assume that our Hilbert spaces are infinite dimensional. 
By \cite{BYBHU} they form the class of models of a complete 
theory which is $\kappa$-categorical for all infinite $\kappa$, 
and admits elimination of quantifiers. 

This approach can be naturally extended to complex Hilbert spaces, 
$$
(\{ B_n\}_{n\in \omega} ,0,\{ I_{mn} \}_{m<n} ,
\{ \lambda_c \}_{c\in\mathbb{C}}, +,-,\langle \rangle_{Re} , \langle \rangle_{Im} ). 
$$
We only extend the family 
$\lambda_{r}: B_m\rightarrow B_{km}$, $r\in \mathbb{R}$, 
to a family $\lambda_{c}: B_m\rightarrow B_{km}$, $c\in \mathbb{C}$, 
of scalar products by $c\in\mathbb{C}$, with $k$ 
the unique integer satisfying $k\ge 1$ and $k-1 \le |c|<k$. 
We also introduce $Re$- and $Im$-parts of the inner product.

\section{Metric groups and amenability}   
Although closed subgroups of amenable 
locally compact groups are amenable, 
amenability is not preserved under elementary extensions. 
For example there are locally finite countable groups 
having elementary extensions containing free groups 
(i.e. in the discrete case amenability is not axiomatizable). 
In this section we apply the description of amenable topological groups
found by F.M. Schneider and A. Thom in \cite{SchneiderThom}  in order to axiomatize in $L_{\omega_1 \omega}$
amenability for metric groups which are continuous structures. 
In fact we will see that this property is $\mathsf{sup} \bigvee \mathsf{inf}$ 
in terms of \cite{Gold} 
\footnote{in the discrete case this was observed in \cite{Goldbring}}. 
This kind of axiomatization is essential in model-theoretic forcing, 
see Proposition 2.6 in \cite{Gold}.     
Our results imply that (typical) elementary substructures 
of (non-) amenable groups are (non-) amenable. 
In particular these properties are bountiful, see \cite{saszaArX}. 

Let $G$ be a topological group, $F_1 ,F_2\subset G$ are finite and $U$ be an identity neighbourhood. 
Let $R_U$ be a binary relation defined as follows: 
$$
R_U=\{ (x,y) \in F_1\times F_2 : yx^{-1} \in U \} . 
$$ 
This relation defines a bipartite graph on $(F_1 , F_2 )$. 
Let 
$$
\mu (F_1 ,F_2 ,U) = |F_1 | - \mathsf{sup} \{ |S| - |N_R (S)| : S\subseteq F_1 \} , 
 $$ 
where $N_R (S) = \{ y\in F_2 : (\exists x\in S) (x,y)\in R_U \}$. 
By Hall's matching theorem this value is the {\em  matching number} of the graph $(F_1 , F_2 , R_U )$. 
Theorem 4.5 of \cite{SchneiderThom} 
gives the following description of amenable topological groups. 
\begin{quote} 
{\em Let $G$ be a Hausdorff topological group. The following are equivalent. \\ 
(1) $G$ is amenable. \\
(2) For every $\theta \in (0,1)$, every finite subset $E\subseteq G$, and every identity neighbourhood $U$, 
there is a finite non-empty subset $F\subseteq G$ such that 
$$
\forall g\in E (\mu (F,gF,U) \ge \theta |F|). 
$$ 
(3) There exists $\theta \in (0,1)$ such that for every finite subset $E\subseteq G$, and every identity neighbourhood $U$, 
there is a finite non-empty subset $F\subseteq G$ such that 
$$
\forall g\in E (\mu (F,gF,U) \ge \theta |F|). 
$$
}
\end{quote} 
It is worth noting here that when an open neighbourhood $V$ contains $U$ 
the number  $\mu (F,gF,U)$ does not exceed $\mu (F,gF,V)$. 
In particular in the formulation above we may consider neighbourhoods $U$ 
from a fixed base of identity neighbourhoods. 
For example in the case of a continuous structure $(G, d)$ with an invariant $d$ 
we may take all $U$ in the form of metric balls 
$B_{<q} = \{ x : d(1,x) < q\}$, $q\in \mathbb{Q}\cap (0,1)$.  
It is also clear that we can restrict all $\theta$ by rational ones. 
From now on we work in this case. 

In Lemma \ref{amenlem} we consider the Schneider-Thom theorem in the formulation where all $U$ are closed balls 
$B_q = \{ x : d(1,x) \le q\}$, $q\in \mathbb{Q}\cap (0,1)$.  
Notice that the corresponding versions of statement (2) above 
are equivalent for $U$ of the form $B_{<q}$ and of the form $B_q$. 
Indeed this follows from the observation that 
 $\mu (F,gF,B_{<q})\le \mu (F,gF,B_q )$ and 
  $\mu (F,gF,B_q )\le \mu (F,gF,B_{<r})$ for $q<r$.

\begin{lemma} \label{amenlem} 
Given $k\in \mathbb{N}$ and rational numbers 
$q, \theta\in (0,1)$ there is a quantifier free formula 
$\phi_{k,q, \theta} (\bar{x},y)$ depending on variables 
$x_1 ,\ldots ,x_k$ and $y$ such that in the structure 
$(G,d)$ the $0$-statement 
$\phi_{k,q, \theta} (\bar{x},y) \le 0$  is equivalent 
to the condition that $x_1 ,\ldots ,x_k$ form a set $F$ 
with $\mu (F,yF,B_q ) \ge \theta |F|$.  

Moreover the identity function is a continuity modulus of $y$ 
in $\phi_{k,q, \theta} (\bar{x},y)$. 
\end{lemma} 

\begin{proof} 
To guarantee the inequality $\mu (F,gF,B_q ) \ge \theta |F|$  
for an $F= \{ f_1 ,\ldots , f_k \}$ 
we only need to demand that for every 
$S\subseteq F$ the following inequality holds: 
$$ 
| S| - k + \theta \cdot k \le |N_R (S)| , 
$$
where  $N_R (S)$ is defined with respect to $(F, gF)$ 
and $U=B_q$. 

To satisfy this inequality we will use the observation that 
when $S'\subseteq gF$ and $\rho$ is a function 
$S' \rightarrow S$ such that 
$\mathsf{max} \{ d (g f ,\rho (g f )) : g f \in S' \} \le q$ then $|S'| \le |N_R (S ))|$. 
Let us assume that $S$ corresponds to some tuple 
$x_{i_1}, \ldots , x_{i_l}$ of elements of 
$\{ x_1 ,\ldots ,x_k \}$, 
the subset $S'$ corresponds to some tuple of terms  
from $\{ y x_1 ,\ldots ,y x_k \}$ (recovered by $\rho^{-1}$) 
and let 
$$ 
dist_{S,S',\rho }(x_1 , \ldots x_k, y) = \mathsf{max} \{ d (yx_i ,\rho (yx_i )) : yx_i \in S' \} 
$$ 
(a $\mathsf{max}$-formula of continuous logic).  
Then the statement formalizing $|S'| \le |N_R (S ))|$ can be expressed that the formula 
$dist_{S,S',\rho }(x_1 , \ldots x_k, y)$
takes value $\le q$ with respect to the realization 
of $\bar{x}$ by the tuple $f_1 , \ldots f_k$ 
and $y$ by $g$.   

Thus the following formula $\phi_{k,q, \theta} (\bar{x},y)$ satisfies the statement of the lemma: 
$$
\mathsf{max}_{S\subseteq \{ x_1, \ldots , x_k \}} 
\mathsf{min} \{    dist_{S,S',\rho }(x_1 , \ldots x_k, y)
: S' \subseteq \{ y x_1 ,\ldots ,y x_k \} 
\mbox{ , }
\rho :S' \rightarrow S \mbox{ , }  
$$
$$ 
| S| - k + \theta \cdot k \le |S'| \} \dot{-} q . 
$$ 

To see the last statement of the lemma it suffices to notice 
that the identity function is a continuity modulus of $y$ in each  $dist_{S,S',\rho }(x_1 , \ldots x_k, y)$. 
The latter follows from the definition of 
$dist_{S,S',\rho }(x_1 , \ldots x_k, y)$ and the fact that 
$d$ is invariant and satisfies the triangle inequality.  
\end{proof}

\bigskip

\begin{theorem} \label{thmamen} 
The class of all amenable groups which are continuous structures 
with invariant metrics, is axiomatizable by all $L_{\omega_1 \omega}$-statements 
of the following form: 
$$ 
\mathsf{sup}_{y_1 \ldots y_l} 
\bigvee \{ 
\mathsf{inf}_{x_1 \ldots x_k } 
\mathsf{max} \{ \phi_{k,q, \theta} (\bar{x},y_i) : 1\le i\le l\} : k\in \omega \} \le 0, 
$$ 
$$ \mbox{ where } \theta , q \in \mathbb{Q}\cap (0,1) \mbox{ and } 
 l \in \omega .  
$$ 
In particular every first order elementary substructure 
of a continuous structure which is an amenable group is also an amenable group.  
\end{theorem} 

\begin{proof} 
The tuple $y_1 ,\ldots ,y_l$ consists of all free variables of 
the formula 
$$
\mathsf{inf}_{x_1 \ldots x_k } \mathsf{max} \{ \phi_{k,q, \theta} (\bar{x},y_i) : 1\le i\le l\} . 
$$ 
By the last statement of Lemma \ref{amenlem} the identity function is a continuity modulus of 
each $y_i$ in this formula. 
This implies the same statement concerning the infinite disjunction in the formulation. 
Thus the formula in the formulation belongs to $L_{\omega_1 \omega}$. 

By Theorem 4.5 of \cite{SchneiderThom} and the discussion after 
the formulation of that theorem above we see  
that all amenable groups satisfy the statements in the formulation. 

Let us prove the contrary direction. 
Let $(G,d)$ satisfy the axioms from the formulation. 
We want to apply condition (2) of Theorem 4.5 of \cite{SchneiderThom} in the case 
of balls $B_q$.  
Fix $\theta, q$ and $E$ as in the formulation so that $E=\{ g_1 ,\ldots ,g_l \}$. 
Choose $q' \in \mathbb{Q}\cap (0,1)$ so that $q'<q$. 
Since $(G,d)$ satisfies 
$$ 
\mathsf{sup}_{y_1 \ldots y_l} 
\bigvee \{ 
\mathsf{inf}_{x_1 \ldots x_k } 
\mathsf{max} \{ \phi_{k,q', \theta} (\bar{x},y_i) : 1\le i\le l\} : k\in \omega \} \le 0, 
$$ 
we find $f_1 ,\ldots ,f_k \in G$ so that 
$$ 
\mathsf{max} \{ \phi_{k,q', \theta} (\bar{f},g_i) : 1\le i\le l\}  < q-q'.  
$$ 
By the definition of the formula $\phi_{k,q', \theta} (\bar{f},g_i)$ 
we see that condition (2) of Theorem 4.5 of \cite{SchneiderThom} 
holds for $\theta$, $U=B_q$ and $E$. 
This proves that $(G,d)$ is amenable. 

To see the last statement of the theorem assume that 
$(G,d)$ is amenable and $(G_1 ,d)\preceq (G,d)$. 
We repeat the argument of the previous paragraph for $E\subseteq G_1$. 
Then having found $f_1 ,\ldots ,f_k \in G$ as above 
we can apply the definition of an elementary substructure 
in order to obtain $f'_1 ,\ldots ,f'_k \in G_1$ so that 
$$ 
\mathsf{max} \{ \phi_{k,q', \theta} (\bar{f}',g_i) : 1\le i\le l\}  < q-q'.  
$$ 
The rest is clear. 
\end{proof} 

\bigskip

To have a similar theorem for non-amenability we need the following lemma. 

\begin{lemma} \label{namenlem} 
Given $k\in \mathbb{N}$ and rational numbers 
$q, \theta\in (0,1)$ there is a quantifier free formula 
$\phi^{-}_{k,q, \theta} (\bar{x},y)$ depending on variables 
$x_1 ,\ldots ,x_k$ and $y$ such that in the structure 
$(G,d)$ the $0$-statement 
$\phi_{k,q, \theta} (\bar{x},y) \le 0$  is equivalent 
to the condition that $x_1 ,\ldots ,x_k$ form a set $F$ 
with $\mu (F,yF,B_{<q} ) < \theta |F|$.  

Moreover the identity function is a continuity modulus of $y$ 
in $\phi^{-}_{k,q, \theta} (\bar{x},y)$. 
\end{lemma} 

\begin{proof} 
To guarantee the inequality $\mu (F,gF,B_{<q} ) < \theta |F|$  
for an $F= \{ f_1 ,\ldots , f_k \}$ 
we only need to demand that there is 
$S\subseteq F$ such that  
$$ 
|N_R (S)| < | S| - k + \theta \cdot k , 
$$
where  $N_R (S)$ is defined with respect to $(F, gF)$ 
and $U=B_{<q}$.   

To formalize this inequality we will use the observation that 
when $S'\subseteq gF$ satisfies the inequality 
$$
q \le \mathsf{inf} \{ d (g f , f' )) : f'\in S\mbox{ and }
g f \not\in S' \} , 
$$ 
then $|S'| \ge |N_R (S ))|$. 
Thus if we associate to $S$ and $S'$ some tuple of elements of 
$\{ x_1 ,\ldots ,x_k \}$ and some tuple of terms  
from $\{ y x_1 ,\ldots ,y x_k \}$ respectively,  
then the statement $|S'| \ge |N_R (S ))|$ can be expressed 
that the formula 
$$ 
\overline{dist}_{S,S'}(x_1 , \ldots x_k, y) = \mathsf{inf} \{ d (yx_i ,x_j )) : x_j \in S \mbox{ and } yx_i \not\in S' \} 
$$ 
takes value $\ge q$ with respect to the realization 
of $\bar{x}$ by the tuple $f_1 , \ldots f_k$ 
and $y$ by $g$. 

Thus the following formula $\phi^{-}_{k,q, \theta} (\bar{x},y)$ satisfies the statement of the lemma. 

$$
\mathsf{min}_{S\subseteq \{ x_1, \ldots , x_k \}} 
\mathsf{min} \{  q\dot{-}  
\overline{dist}_{S,S' }(x_1 , \ldots x_k, y) : 
S' \subseteq \{ y x_1 ,\ldots ,y x_k \} 
\mbox{ , }  
$$
$$ 
|S'| < | S| - k + \theta \cdot k  \} . 
$$ 
The statement that the identity function is a continuity modulus of $y$ in the formula $\phi^{-}_{k,q, \theta} (\bar{x},y)$ 
follows from the definition of this formula and the 
assumption that $d$ is an invariant metric.  
\end{proof}

\bigskip

\begin{theorem} 
The class of all non-amenable groups which are continuous structures 
with invariant metrics, is axiomatizable by all $L_{\omega_1 \omega}$-statements 
of the following form: 
$$ 
\bigvee_q \bigvee_{l\in \omega}\{ \mathsf{inf}_{y_1 \ldots y_l} 
\bigwedge \{ 
\mathsf{sup}_{x_1 \ldots x_k } 
\mathsf{min} \{ \phi^{-}_{k,q, \theta} (\bar{x},y_i) : 1\le i\le l\} : k\in \omega \}: q \in \mathbb{Q}\cap (0,1)\} \le 0, 
$$ 
$$ \mbox{ where } \theta \in \mathbb{Q}\cap (0,1).  
$$ 
Moreover every subset $A$ of a continuous structure $(G,d)$  
which is a non-amenable group is contained in a  first order 
elementary substructure of $(G,d)$ of density character $\le |A|+\aleph_0$   
which also a non-amenable group.  
\end{theorem}

\begin{proof} 
As in the proof of Theorem \ref{thmamen} one can show that the 
formula from the formulation of the theorem satisfies the requirements to 
be an $L_{\omega_1 \omega}$-formula of continuous logic. 

By Theorem 4.5 (3) of \cite{SchneiderThom} and the discussion after 
the formulation of that theorem above we see  
that all non-amenable groups satisfy the statements in the formulation. 

Let us prove the contrary direction. 
Let $(G,d)$ satisfy the axioms from the formulation. 
To apply condition (3) of Theorem 4.5 of \cite{SchneiderThom} 
fix $\theta \in \mathbb{Q} \cap (0,1)$. 
For every $\varepsilon >0$ one can choose $q$ and $E=\{ g_1 ,\ldots ,g_l \}$  
so that $(G,d)$ satisfies 
$$  
\bigwedge \{ 
\mathsf{sup}_{x_1 \ldots x_k } 
\mathsf{min} \{ \phi^{-}_{k,q, \theta} (\bar{x},g_i) : 1\le i\le l\} : k\in \omega \} \le \varepsilon .  
$$ 
This obviously means that 
$$  
\bigwedge \{ 
\mathsf{sup}_{x_1 \ldots x_k } 
\mathsf{min} \{ \phi^{-}_{k,q-\varepsilon , \theta} (\bar{x},g_i) : 1\le i\le l\} : k\in \omega \} \le  0 .  
$$ 
By the definition of the formula $\phi^{-}_{k,q-\varepsilon , \theta} (\bar{x},g_i)$ 
we see that condition (3) of Theorem 4.5 of \cite{SchneiderThom} 
does not hold for $\theta$, $U=B_{<q-\varepsilon}$ and $E$. 
This proves that $(G,d)$ is not amenable. 

To see the last statement of the theorem assume that 
$(G,d)$ is not amenable. 
Then applying Theorem 4.5 (2) of \cite{SchneiderThom} 
we find $\theta \in (0,1)$, a finite subset $E=\{ g_1 ,\ldots ,g_l \}\subseteq G$, and  
an identity neighbourhood $U = B_{<q}$, 
such that for every finite non-empty subset $F\subseteq G$  
$$
\exists g\in E (\mu (F,gF,U) < \theta |F|). 
$$
Let $(G_1 ,d)\preceq (G,d)$ and $E\subseteq G_1$. 
Then by Lemma \ref{namenlem} 
$$ 
G_1 \models \bigwedge \{ \mathsf{sup}_{x_1 \ldots x_k} 
\mathsf{min} \{ \phi^{-}_{k,q, \theta} (\bar{x},g_i) : 1\le i\le l\}: k\in \omega \} \le 0.  
$$ 
By Theorem 4.5 (2) of \cite{SchneiderThom} the group $(G_1,d)$ 
is not amenable. 
It remains to note that given $A$ as in the formulation above 
by the L\"{o}wenheim-Skolem Theorem 
for continuous logic (\cite{BYBHU}, Proposition 7.3)   
the substructure $G_1$ can be chosen of density character $\le |A|+\aleph_0$ 
with $A\cup E \subset G_1$. 
\end{proof}

\bigskip


\section{Negating (T)} 

It is well-known that a locally compact group with property 
{\bf (T)} of Kazhdan is amenable if and only if it is compact. 
Thus axiomatization of property {\bf (T)} (non-{\bf (T)}) 
is natural in the context of axiomatization of amenability 
(at least for locally compact groups). 
In this section we apply continuous logic to {\bf (T)}/non-{\bf (T)}. 
Our results are partial. 
On the one hand they are restricted to the class of locally compact groups 
and on the other one they mainly concern property non-{\bf (T)}.

\subsection{  Introduction } 
Let a topological group $G$ have a continuous unitary 
representation on a complex Hilbert space ${\bf H}$.  
A closed subset $Q\subset G$ 
has an {\bf $\varepsilon$-invariant unit vector} in ${\bf H}$ if 
$$ 
\mbox{ there exists }v\in {\bf H} \mbox{ such that } 
\mathsf{sup}_{x\in Q} \parallel x\circ v - v\parallel < \varepsilon
\mbox{ and } \parallel v\parallel =1.  
$$ 
A closed subset $Q$ of the group $G$ is called a {\bf Kazhdan set} 
if there is $\varepsilon >0$ with the following property: 
for every unitary representation of $G$ on a Hilbert space 
where $Q$ has an $\varepsilon$-invariant unit vector  
there is a non-zero $G$-invariant vector.  
The following statement is Proposition 1.1.4 from 
\cite{BHV}. 
\begin{quote} 
Let $G$ be a topological group. 
The pair $(G,\sqrt{2} )$ is Kazhdan pair, i.e. 
if a unitary representation of $G$ has a  
$\sqrt{2}$-invariant unit vector then $G$ has a non-zero 
invariant vector.  
\end{quote} 
If the group $G$ has a compact Kazhdan subset then 
it is said that $G$ {\bf has property ${\bf (T)}$ of Kazhdan}. 

Proposition 1.2.1 of \cite{BHV} states that the group $G$ 
has  property  ${\bf (T)}$ of Kazhdan if and only 
if any unitary representation of $G$ which weakly contains 
the unit representation of $G$ in $\mathbb{C}$ 
has a fixed unit vector. 

By Corollary F.1.5 of \cite{BHV} the property that 
the unit representation of $G$ in $\mathbb{C}$ 
is {\bf weakly contained} in a unitary representation $\pi$ 
of $G$ (this is denoted by $1_{G} \prec \pi$) is 
equivalent to the property that for every compact subset 
$Q$ of $G$ and every $\varepsilon >0$  
the set $Q$ has an $\varepsilon$-invariant 
unit vector with respect to $\pi$.

The following example shows that in the first-order logic 
property ${\bf (T)}$ is not elementary: there are two groups 
$G_1$ and $G_2$ which satisfy the same sentences of the 
first-order logic but $G_1 \models {\bf (T)}$ and 
$G_2 \not\models {\bf (T)}$. 

\bigskip

{\bf Example.} 
 Let $n>2$. 
According Example 1.7.4 of \cite{BHV} the group 
$SL_n (\mathbb{Z})$ has property ${\bf (T)}$. 
Let $G$ be a countable elementary extension of 
$SL_n (\mathbb{Z})$ which is not finitely generated. 
Then by Theorem 1.3.1 of \cite{BHV} the group 
$G$ does not have ${\bf (T)}$.  

\bigskip

The main result of this section, Theorem \ref{nT},  
shows that in the context of continuous logic 
the class of unitary representations of locally compact groups 
with property non-{\bf (T)} can be viewed as the union of axiomatizable classes.

\subsection{Unitary representations in continuous logic}
We apply methods announced in the introduction. 
In order to treat axiomatizability question  in the class of 
locally compact groups satisfying some uniform version 
of property non-${\bf (T)}$ we need the preliminaries 
of continuous model theory of Hilbert spaces from Section \ref{hs}. 
Moreover since we want to consider unitary representations of 
metric groups $G$ in continuous logic we should 
fix continuity moduli for the corresponding binary functions 
$G\times B_n \rightarrow B_n$ induced by $G$-actions on 
metric balls of the corresponding Hilbert space.

\begin{remark} 
Continuous unitary actions of $G$ on $B_1$ obviously determine their 
extensions to $\bigcup \{ B_i : i>1 \}$: 
$$
g ( r \cdot x ) = r \cdot g(x) \mbox{ where } x\in B_1 \mbox{ and } 
r\cdot x \in B_n . 
$$ 
Thus a continuity modulus, say $F$, for the corresponding 
function $G\times B_1 \rightarrow B_1$ 
can be considered as a family of continuity 
moduli for $G\times B_i \rightarrow B_i$ as follows: 
$$ 
F_i (\varepsilon ) =  F (\frac{\varepsilon }{i} ). 
$$
Using this observation we simplify the approach by    
considering only continuity moduli 
for $G\times B_1 \rightarrow B_1$. 
When we fix such $F$ we call the corresponding 
continuous unitary action of $G$ an $F$-{\bf continuous} action. 
\end{remark} 

We now define a uniformly  
continuous versions of the notion of a Kazhdan set.

\begin{definition} 
Let $G$ be a metric group of diameter $\le 1$ 
which is a continuous structure in the language  
$(d,\cdot ,^{-1},1)$.  
Let $F$ be a continuity modulus for the $G$-variable of 
continuous functions $G \times B_1 \rightarrow B_1$. 

We call a closed subset $Q$ of the group $G$ 
an $F$-{\bf Kazhdan set} if there is $\varepsilon$ 
with the following property: 
every $F$-continuous unitary representation of $G$ 
on a Hilbert space with $(Q,\varepsilon )$-invariant unit vectors 
also has a non-zero invariant vector.  
\end{definition} 

It is clear that for any continuity 
modulus $F$ a subset $Q\subset G$ 
is $F$-Kazhdan if it is Kazhdan. 
We will say that $G$ has property 
$F$-{\bf non-(T)} if $G$ does not have 
a compact $F$-Kazhdan subset.

To study such actions in continuous logic  
let us consider a class of many-sorted continuous 
metric structures which consist of groups $G$ 
together with metric structures 
of complex Hilbert spaces  
$$ 
(d, \cdot , ^{-1}, 1 ) \cup 
(\{ B_n\}_{n\in \omega} ,0,\{ I_{mn} \}_{m<n} ,
\{ \lambda_c \}_{c\in\mathbb{C}}, +,-,\langle \rangle_{Re} , \langle \rangle_{Im} ).
$$
Such a structure  also contains 
a binary operation $\circ$ of an action which is 
defined by a family of appropriate maps  
$G \times B_m \rightarrow B_{m}$ 
(in fact $\circ$ is presented by a sequence of functions 
$\circ_m$ which agree with respect to all $I_{mn}$). 
When we add the obvious continuous $\mathsf{sup}$-axioms that 
the action is linear and unitary, we obtain an axiomatizable 
class $\mathcal{K}_{GH}$. 
Given unitary action of $G$ on ${\bf H}$ we denote 
by $A(G,{\bf H})$ the member of $\mathcal{K}_{GH}$ 
which is obtained from this action.  
When we fix a continuity modulus, say 
$F$, for the $G$-variables of the operation $G \times B_1 \rightarrow B_{1}$ 
 we denote by $\mathcal{K}_{GH}(F )$ the 
corresponding subclass of $\mathcal{K}_{GH}$. 

\begin{definition} \label{K-delta}  
{\bf  The class}  
$\bigcup \{ \mathcal{K}_{\delta}(F): \delta \in (0,1) \cap \mathbb{Q}\}$.  
Let 
$\mathcal{K}_{\delta}(F)$ 
be the subclass of $\mathcal{K}_{GH}(F )$ 
axiomatizable by all axioms of the following form
$$ 
\mathsf{sup}_{x_1 ,...,x_k \in G} \mathsf{inf}_{v\in B_{m}} 
\mathsf{sup}_{x\in \bigcup x_{i}K_{\delta}} 
\mathsf{max}(\parallel x\circ v - v\parallel \dot- \mbox{ } {1\over n}, 
\arrowvert 1- \parallel v\parallel  \arrowvert )=0 ,
$$
$$
\mbox{ where } k,m,n\in\omega \setminus \{ 0\} 
\mbox{ and } K_{\delta} = \{ g\in G: d(1,g)\le \delta \}. 
$$ 
\end{definition} 

It is easy to see that the axiom of 
Definition \ref{K-delta} implies that 
each finite union $\bigcup^{k}_{i = 1} g_i K_{\delta}$ 
has a ${1\over n}$-invariant unit 
vector in ${\bf H}$. 
To see that it can be written 
by a formula of continuous logic note that  
$\mathsf{sup}_{x\in \bigcup x_{i}K_{\delta}}$
can be replaced by $\mathsf{sup}_{x}$ 
with simultaneous including of the quantifier-free 
part together with  
$\mathsf{max}(\delta\dot{-} d(x,x_i ) : 1\le i\le k)$
into the corresponding $\mathsf{min}$-formula. 

In fact the following theorem shows that in the class  of 
 locally compact metric groups condition 
 $F$-non-{\bf (T)} is a union of axiomatizable classes. 

\begin{theorem} \label{nT} 
Let $F$ be a continuity modulus for the $G$-variable 
of continuous functions $G \times B_1 \rightarrow B_1$. 

(a) In the class of all unitary $F$-representations 
of locally compact metric groups the condition of weak 
containment of the unit representation $1_{G}$ coincides  
with the condition that the corresponding structure  
$A(G,{\bf H})$ belongs to 
$\bigcup \{ \mathcal{K}_{\delta}(F): \delta \in (0,1) \cap \mathbb{Q}\}$. 

(b) In the class of all unitary $F$-representations 
of locally compact metric groups the condition of 
witnessing $F$-non-{\bf(T)} corresponds to a union 
of axiomatizable classes of structures of the form 
$A(G,{\bf H})$. 
\end{theorem} 

\begin{proof} 
(a) Let $G$ be a locally compact metric group 
and let the ball 
$K_{\delta}= \{ g\in G: d(g,1)\le \delta \}\subseteq G$ 
be compact. 
If a unitary $F$-representation of $G$ 
weakly contains the unit representation 
$1_{G}$, then considering it as a structure 
$A(G,{\bf H})$ we see that this structure belongs to 
$\mathcal{K}_{\delta}(F)$. 

On the other hand if some structure of the form 
$A(G,{\bf H})$ belongs to 
$\mathcal{K}_{\varepsilon}(\mathcal{F})$, 
then assuming that $\varepsilon \le\delta$ 
we easily see that the corresponding representation 
weakly contains $1_{G}$. 
If $\delta <\varepsilon$, then $K_{\varepsilon}$ 
may be non-compact. 
However since $K_{\delta} \subseteq K_{\varepsilon}$ 
any compact subset of $G$ belongs to a finite union 
of sets of the form $xK_{\varepsilon}$. 
Thus the axioms of $\mathcal{K}_{\varepsilon}(F)$ 
state that the corresponding structure $A(G,{\bf H})$ 
defines a representation weakly containing $1_{G}$. 

(b) The condition 
$$ 
\mathsf{sup}_{v\in B_1}  \mathsf{inf}_{x\in G}  
(1 \dot{-} (\parallel x\circ v - v\parallel + |1 -\parallel v\parallel |)) \le  0  
$$ 
(we call it {\bf NIV}) 
obviuosly implies that $G$ does not have invariant unit vectors. 
For the contrary direction we use the fact mentioned in the introduction of Section 4 that 
the absence of $G$-invariant unit vectors implies that $G$ 
does not have $\sqrt{2}$-invariant unit vectors.  
In particular for every $v\in B_1$ there is $x\in G$ 
such that 
$$
\sqrt{2} \parallel v\parallel\le \parallel x\circ v - v \parallel .
$$
This obviously implies {\bf NIV}. 

Adding {\bf NIV} to each $\mathcal{K}_{\delta}(F)$ we 
obtain axiomatizable classes as in the statement of (b). 
Below we call it $\mathcal{K}^{{\bf NIV}}_{\delta}(F)$.  
\end{proof}

\section{Comments} 

{\bf (I)} It is clear that the classes 
$\bigcup \{ \mathcal{K}_{\delta}(F): \delta \in (0,1) \cap \mathbb{Q}\}$
and 
$\bigcup \{ \mathcal{K}^{{\bf NIV}}_{\delta}(F): \delta \in (0,1) \cap \mathbb{Q}\}$ 
can be considered without the restriction of local compactness. 
However it does not look likely that then they axiomatize 
weak containment of the unit representation $1_{G}$ 
or witnessing non-{\bf (T)}. \\  
{\bf (II)} 
In spite of axiomatizability issues in our paper it is still not clear how large the class of locally compact non-compact groups with bi-invariant merics and property {\bf (T)}. 
This is especially interesting in the case of connected groups,    
since some standard Lie group examples do not admit compatible bi-invariant metrics. \\ 
{\bf (III)} 
Using the method of Proposition 1.2.1 from \cite{BHV}
one can show that if for every compact subset $Q$ of a locally compact 
metric group $(G,d)$ and every $\varepsilon >0$ 
there is an expansion of $G$ to a structure from 
$\mathcal{K}_{GH}(F)$ 
with a $(Q,\varepsilon )$-invariant unit vector 
but without non-zero invariant vectors, then 
$G$ has a unitary $F$-representation  
which belongs to 
$\mathcal{K}^{{\bf NIV}}_{\delta}(F)$ for appropriate 
$\delta \in (0,1) \cap\mathbb{Q}$.  \\ 
{\bf (IV)} 
Since the class of locally compact metric groups is not axiomatizable
\footnote{an ultraproduct of compact metric groups is not necessarily locally compact}, 
the subclasses of $\mathcal{K}_{\delta}(F)$ and 
$\mathcal{K}^{{\bf NIV}}_{\delta}(F)$ appearing in Theorem \ref{nT} 
are only {\em relatively axiomatizable}. 
On the other hand they have some nice properties of axiomatizable classes. 
For example the following statement holds. 

\begin{proposition} 
Any elementary substructure of any 
$A(G,{\bf H})\in \bigcup \{ \mathcal{K}^{{\bf NIV}}_{\delta}(F): \delta \in (0,1) \cap\mathbb{Q} \}$ 
with a locally compact $G$ 
also belongs to $\bigcup \{ \mathcal{K}^{{\bf NIV}}_{\delta}(F): \delta \in (0,1) \cap\mathbb{Q} \}$ 
and is of the form $A(G_0 ,{\bf H}_0 )$, 
where $G_0 \preceq G$, ${\bf H}_0 \preceq {\bf H}$ 
and $G_0$ is locally compact.  
\end{proposition} 

In the proof we use an additional tool 
from model theory. 
Let $M$ be a continuous metric structure. 
A tuple $\bar{a}$ from $M^n$ is {\em algebraic} in $M$ over 
$A$ if there is a compact subset $C\subseteq M^n$ such that 
$\bar{a}\in C$ and the distance predicate $\mathsf{dist}(\bar{x},C)$ 
is definable (in the sense of continuous logic, \cite{BYBHU}) 
in $M$ over $A$. 
Let $\mathsf{acl}(A)$ be the set of all elements algebraic over $A$. 
In continuous logic the concept of algebraicity is 
parallel to that in traditional model theory 
(see Section 10 of \cite{BYBHU}).  

\begin{proof} 
Let 
$M\in \bigcup \{ \mathcal{K}^{{\bf NIV}}_{\delta}(F): \delta \in (0,1) \cap\mathbb{Q} \}$ and $G$ be the group sort of $M$.
Choose $\delta >0$ so that 
the $\delta$-ball 
$K = \{ g\in G: d(g,1)\le \delta \}$ 
in $G$ is compact. 
Note that since the condition $d(g,1)\le \delta$ 
defines a totally bounded complete subset 
in any elementary extension of $G$, the set $K$ 
is a definable subset of $acl(\emptyset )$.  

Let $M_0 \preceq M$ and $G_0$ be the sort of $M_0$ 
corresponding to $G$. 
It remains to verify that for any compact subset 
$D\subset G_0$ and any $\varepsilon >0$ 
the representation $M_0$ always has 
a $(D,\varepsilon )$-invariant unit vector. 
To see this note that since $G_0 \prec G$ and $K$ is 
compact and algebraic, 
the ball $\{ g\in G_0: d(g,1)\le \delta \}\subset G_0$ 
is a compact neighborhood of 1 which coincides with $K$. 
In particular $D$ is contained in a finite union of 
sets of the form $gK$. 
The rest follows  from the conditions that 
$M_0 \in \mathcal{K}_{\delta}(F)$ and $G_0 \prec G$. 
\end{proof}

We do not know if the statement of this proposition  
holds without the assumtion that $G$  
locally compact. \\  
{\bf (V)}
 In Section 8.4 of \cite{pestov16} 
it is proved that if $\Gamma$ is a discrete group with 
property {\bf (T)}, then the direct power $\Gamma^{\omega}$ 
also has property {\bf (T)} as a topological group. 
On the other hand the topological group $\Gamma^{\omega}$ is 
a continuous metric group under the obvious metric. 
There is also a certain class of 'trivial examples' of 
non-locally compact groups with bi-invariant metrics 
that have property {\bf (T)}. 
Namely, there are abelian metrizable groups that admit no 
non-trivial unitary representations, i.e. satisfying 
property {\bf (T)}. 
Such an example can be found in \cite{doucha}.  
Since it is extremely amenable it does not admit non-trivial 
unitary representations. 
The author does not know other non-locally compact groups which are continuous metric groups with property {\bf (T)}. 
In particular are there non-trivial connected examples? 
This remark originally appeared in a discussion with Michal Doucha and then was extended by the referee.  \\  
{\bf (VI)}  \label{QT} 
The author thinks that the following question 
is basic in this topic.   
\begin{quote}  
{\em Let $(G,d)$ be a metric group which is a continuous structure. 
Assume that property {\bf (T)} holds in $G$. 
Does every elementary substructure of $(G,d)$ satisfy {\bf (T)}? }   
\end{quote}  
According the previous remark it looks reasonable to start with the discrete case: 
\begin{quote}
{\em Does an elementary substructure of a discrete group 
with property {\bf (T)} also have  
property {\bf (T)}? } 
\end{quote} 
 It is natural to consider this question in 
the case of linear groups, where property {\bf (T)} 
and elementary equivalence are actively studied, see 
\cite{Bunina} and \cite{EJZ}. \\ 
{\bf (VII)} 
Property {\bf FH} states that every action of $G$ 
by affine isometries on a Hilbert space has a fixed point. 
It is equivalent to property {\bf (T)} for 
$\sigma$-compact locally compact groups. 
Axiomatization of {\bf FH} is studied in 
arXiv paper \cite{saszaArX}. 
Since in this case unbounded actions appear, 
the approach is different there. \\ 
{\bf (VIII)} 
One of the definitions of non-amenability says that 
a topological group is  non-amenable if 
there is a locally convex topological 
vector space $V$ and a continuous affine 
representation of $G$ on $V$ such that  
some non-empty invariant convex compact subset $K$ of $V$ 
does not contain a $G$-fixed point (\cite{BHV}, Theorem G.1.7).  
If we restrict ourselves just by linear representations 
on normed/metric vector spaces we obtain a property 
which is stronger than non-amenability. 
We call it {\bf strong non-FP}. 
The  paper \cite{saszaArX} 
contains some results showing that the approach 
to non-{\bf (T)} presented in Sections 4 and 5 
can be applied to strong non-FP too. 
We would also mention that 
this arXiv paper also considers the class of groups 
which are not extremely amenable in some uniform way.  

\subsection*{Acknowledgments}

The research was partially supported by Polish National Science Centre grant DEC2011/01/B/ST1/01406. 

The author is grateful to the referee for helpful remarks. 
For example some referee's observations are included into 
Comments {\bf (I)} - {\bf (VIII)}.  

Section 4 of this paper is a version of Section 2 of 
\cite{saszaArX} (which was split).  
The author is grateful to the referee of \cite{saszaArX} 
for an advice concerning the proof of Theorem \ref{nT}.

\bigskip 

Department of Applied Mathematics, Silesian University of Technology\\  
ul.Kaszubska 23, 44-101 Gliwice, Poland \\ 
E-mail:  Aleksander.Iwanow@polsl.pl

\end{document}